\documentclass[journal,onecolumn,12pt]{IEEEtran}
\usepackage{amsmath}
\usepackage{amsfonts}
\usepackage{amssymb}
\usepackage{amsthm}
\usepackage{bbm}
\usepackage{float}
\usepackage{makecell}
\usepackage{url}
\usepackage{amscd}
\usepackage{mathrsfs}
\usepackage{graphicx}
\usepackage[margin=0.5in]{geometry} 

\newcommand{\isep}{\mathrel{{.}\,{.}}\nobreak}

\usepackage{setspace}
\usepackage[OT2,T1]{fontenc}
\DeclareSymbolFont{cyrletters}{OT2}{wncyr}{m}{n}
\DeclareMathSymbol{\Sha}{\mathalpha}{cyrletters}{"58}

\newtheorem{theorem}{Theorem}
\newtheorem{corollary}[theorem]{Corollary}
\newtheorem{lemma}[theorem]{Lemma}

\newtheorem{remark}{Remark}

\newcommand{\MP}{\mathrm{MP},y}
\newcommand{\Tr}{{\bf Tr}}
\newcommand{\tr}{\mathrm{Tr}}
\newcommand{\ep}{\varepsilon}
\newcommand{\E}{\mathbb{E}}
\newcommand{\G}{\mathcal{G}_n}
\newcommand{\M}{\mathcal{M}}
\newcommand{\C}{\mathcal{C}}
\newcommand{\Cb}{\C^\bot}
\newcommand{\D}{\mathcal{D}}
\newcommand{\F}{\mathbb{F}}
\renewcommand{\P}{\mathbb{P}}
\newcommand{\db}{d^\bot}

\newcommand{\supp}{\mathrm{supp}}
\renewcommand{\l}{\ell}

\newcommand{\Gl}{G^{(\l)}}
\newcommand{\R}{\mathcal{R}}

\newcommand{\Al}{A_\l}
\newcommand{\Yl}{Y_\l}
\newcommand{\Zl}{Z_\l}
\newcommand{\dl}{\Delta_\l}
\newcommand{\an}{\alpha_n}
\renewcommand{\i}{\mathrm{i}}
\newcommand{\weta}{\widetilde{\eta}}
\renewcommand{\d}{\mathrm{d}}
\newcommand{\hmu}{\hat{\mu}_n}
\newcommand{\hs}{\hat{s}_n}
\newcommand{\pa}{\partial}
\newcommand{\st}{\mathbf{S}_\tau}
\newcommand{\rmp}{\varrho_{\MP}}
\newcommand{\I}{\mathbf{I}}

\title{Convergence Rate of Empirical Spectral Distribution of Random Matrices from Linear Codes}
\author{Chin Hei Chan\thanks{C. Chan is at the Dept. of Mathematics, Hong Kong University of Science and Technology, Clear Water Bay, Kowloon, Hong Kong (email: chchanam@connect.ust.hk).}, Vahid Tarokh\thanks{V. Tarokh is at the Department of Electrical and Computer Engineering, Duke University, Durham, NC, USA (email: vahid.tarokh@duke.edu).} and Maosheng Xiong\thanks{M. Xiong is at the Dept. of Mathematics, Hong Kong University of Science and Technology, Clear Water Bay, Kowloon, Hong Kong (email: mamsxiong@ust.hk). The research of M. Xiong was supported by RGC grant number 16303615 from Hong Kong.}}

\begin{document}
\maketitle
\begin{abstract}
It is known that the empirical spectral distribution of random matrices obtained from linear codes of increasing length converges to the well-known Marchenko-Pastur law, if the Hamming distance of the dual codes is at least 5. In this paper, we prove that the convergence rate in probability is at least of the order $n^{-1/4}$ where $n$ is the length of the code.

\end{abstract}

\begin{keywords}
Group randomness, linear code, dual distance, empirical spectral measure, random matrix theory, Marchenko-Pastur law.
\end{keywords}
\section{Introduction}\label{intro}
Random matrix theory is the study of matrices whose entries are random variables. Of particular interest is the study of eigenvalue statistics of random matrices such as the empirical spectral measure. It has been broadly investigated in a wide variety of areas, including statistics \cite{Wis}, number theory \cite{MET}, economics \cite{econ}, theoretical physics \cite{Wig} and communication theory \cite{TUL}.

Most of the matrix models in the literature are random matrices with independent entries. In a recent series of work (initiated in  \cite{Tarokh0} and developed further in \cite{Tarokh1,Tarokh2,OQBT}), the authors considered a class of sample-covariance type matrices formed randomly from linear codes over a finite field, and proved that if the Hamming distance of the dual codes is at least 5, then as the length of the codes goes to infinity, the empirical spectral distribution of the random matrices obtained in this way converges to the well-known Marchenko-Pastur (MP) law. Since truly random matrices (i.e. random matrices with i.i.d. entries) of large size satisfy this property, this can be interpreted as that sequences from linear codes of dual distance at least 5 behave like random among themselves. This is a new pseudo-random test for sequences and is called a ``group randomness'' property \cite{Tarokh1}. It may have many potential applications.

How fast does the empirical spectral distribution converge to the MP law? This question is interesting in its own rights, and important in applications as one may wish to use linear codes of proper length to generate pseudo-random matrices. Along with proving the convergence in expectation, the authors in \cite{OQBT} obtained a convergence rate of the order $\frac{\log\log n}{\log n}$ where $n$ is the length of the code. This is quite unsatisfactory, as the numerical data showed clearly that the convergence is rather fast with respect to $n$. In this paper, we prove that the convergence rate is indeed at least of the order $n^{-\frac{1}{4}}$ in probability. This substantially improves the previous result. 

To introduce our main result, we need some notation.

Let $\C$ be a linear code of length $n$ and dimension $k$ over the finite field $\F_q$ of order $q$, where $q$ is a prime power. $\C$ is called an $[n,k,d]_q$ linar code for short. The most interesting case is the binary linear codes, corresponding to $q=2$. The dual code $\Cb$ consists of the $n$-tuples in $\F_q$ which are orthogonal to all codewords of $\C$ under the standard inner product. Clearly, $\Cb$ is also a linear code. Denote by $\db$ the Hamming distance of $\Cb$. It is called the \emph{dual distance} of $\C$.

Let $\psi: \F_q \to \mathbb{C}^\times$ be the standard additive character. To be more precise, if $\F_q$ has characteristic $l$, which is a prime number, then $\psi$ is given by $\beta \mapsto \exp\left(2\pi\sqrt{-1}\tr_{q|l}(\beta)/l\right)$, where $\tr_{q|l}$ is the absolute trace mapping from $\F_q$ to $\F_l$. In particular, if $q=l=2$, then the map $\psi: \F_2:=\{0,1\} \to \{-1,1\}$ is defined as $\beta \mapsto (-1)^{\beta}$. We extend $\psi$ component-wise to $\F_q^n$ and obtain the map $\psi: \F_q^n \to (\mathbb{C}^\times)^n$. Denote $\D:=\psi(\C)$.

Denote by $\Phi_n$ a $p \times n$ matrix whose rows are chosen from $\D$ uniformly and independently. This makes the set $\D^p$ a probability space with the uniform probability.

Let $\G$ be the Gram matrix of $X=\frac{1}{\sqrt{n}}\Phi_n$, that is,
\begin{equation}\label{Gram}
\G=XX^*=\frac{1}{n}\Phi_n\Phi_n^*,
\end{equation}
where $A^*$ means the conjugate transpose of the matrix $A$. Let $\mu_n$ be the empirical spectral measure of $\G$, that is,
\begin{equation}\label{mun}
\mu_n=\frac{1}{p}\sum_{j=1}^{p} \delta_{\lambda_j},
\end{equation}
where $\lambda_1\leq \lambda_2\leq \cdots \leq \lambda_p$ are the eigenvalues of $\G$ and $\delta_\lambda$ is the Dirac measure at the point $\lambda$. Note that $\mu_n$ is a random measure, that is, for any interval $\I \subset \mathbb{R}$, the value $\mu_n(\I)$ is a random variable with respect to the probability space $\D^p$. Our main result is as follows.



\begin{theorem} \label{thm}
Assume that $y:=p/n \in (0,1)$ is fixed. If $\db \geq 5$, then
\begin{equation}\label{SD}
|\mu_n(\I)-\rmp(\I)| \prec n^{-\frac{1}{4}}
\end{equation}
uniformly for all intervals $\I \subset \mathbb{R}$. Here $\rmp$ is the empirical spectral measure of the Marchenko-Pastur law whose density function is given by
\begin{equation}\label{mppdf}
\d\rmp(x)=\frac{1}{2\pi xy}\sqrt{(b-x)(x-a)}\mathbbm{1}_{[a,b]}\d x,
\end{equation}
where the constant $a$ and $b$ are defined as
\begin{equation}\label{ab}
a=(1-\sqrt{y})^2, b=(1+\sqrt{y})^2,
\end{equation}
and $\mathbbm{1}_{[a,b]}$ is the indicator function of the interval $[a,b]$.
\end{theorem}

\begin{remark}
The symbol $\prec$ in (\ref{SD}) is a standard notation for ``stochastic domination'' in probability theory (see \cite{Local law} for details). Here it means that for any $\ep >0$ and any $D>0$, there is a quantity $n_0(\ep,D)$, such that whenever $n \geq n_0(\ep,D)$, we have
$$\sup_{\I \, \subset \, \mathbb{R}} \P\left[|\mu_n(\I)-\rmp(\I)| > n^{-\frac{1}{4}+\ep} \right] \leq n^{-D},$$
where $\P$ is the probability with respect to $\D^p$ and the supremum can also be taken over all linear codes $\C$ of length $n$ over $\F_q$ with $\db \geq 5$.
\end{remark}

\begin{remark}
Theorem \ref{thm} is reminiscent of a well-known result of Sidel'nikov (\cite{MacWilliams, Sid}) which states that for any $[n,k,d]$ binary linear code $\C$ with dual distance $\db \geq 3$, one has  
$$|A(z)-\Phi(z)| \leq \frac{9}{\sqrt{\db}}\, .$$
Here $A(z)$ is the normalized cumulative weight distribution of $\C$ and 
$$\Phi(z)=\frac{1}{\sqrt{2 \pi}} \int_{-\infty}^z e^{-t^2/2}\, \mathrm{d} t. $$ 
Hence the ``randomness'' of the weight distribution of $\C$ is ensured if $d^\bot$ is sufficiently large. With this respect Theorem \ref{thm} is a little surprising since the condition $d^\bot \ge 5$ already ensures a fast convergence rate to the MP law (see Equation (\ref{SD})). We emphasize that the condition $d^\bot \ge 5$ is also optimal: the work \cite{Tarokh1} showed that the empirical spectral distribution of random matrices based on binary Simplex (shortened first-order Reed-Muller) codes with $\db=3$ does not converge to the MP law; a similar calculation shows that the empirical spectral distribution of random matrices based on binary first-order Reed-Muller codes with $\db=4$ does not converge to the MP law either. 
\end{remark}


\begin{remark}
It seems quite possible to extend our results to nonlinear codes, where the dual distance $\db$ is defined as in \cite[Chapter 5]{MacWilliams}. In this paper, however, we focus only on linear codes.
\end{remark}

\begin{remark}
For application purposes, from Theorem \ref{thm}, binary linear codes of dual distance 5 with large length and small dimension are desirable as they can be used to generate random matrices efficiently. Here we mention two constructions of binary linear codes with parameters $[2^m-1,2m]$ and dual distance 5. The first family is the dual of primitive double-error correcting BCH codes (\cite{Ding}). The second family of such codes, which includes the well-known Gold codes, can be constructed as follows: Let $f:\F_{2^m} \to \F_{2^m}$ be a function such that $f(0)=0$. Let $n=2^m-1$ and $\alpha$ be a primitive element of $\F_{2^m}$. Define a matrix
$$H_f:=\begin{bmatrix}
1 & \alpha & \alpha^2 & \cdots & \alpha^{n-1}\\
f(1) & f(\alpha) & f(\alpha^2) & \cdots & f(\alpha^{n-1})
\end{bmatrix}.$$
Given a basis of $\F_{2^m}$ over $\F_2:=\{0,1\}$, each element of $\F_{2^m}$ can be identified as an $m \times 1$ column vector in $\F_2$, hence the $H_f$ above can be considered as a binary matrix of size $2m \times n$. Denote by $\C_f$ the binary linear code obtained from $H_f$ as a generator matrix. Note that $\C_f$ has length $2^m-1$ and dimension $2m$. It is known that the dual distance of $\C_f$ is 5 if and only if $f$ is an almost perfect nonlinear (APN) function \cite{Blon, POTT}. Since there are many APNs when $m$ is odd, this provides a general construction of binary linear codes of dual distance 5 which may be of interest for applications. 
\end{remark}

\begin{remark}
Binary linear codes can be used to construct deterministic sensing matrices which satisfy the important ``statistical restricted isometry property'' (\cite{RIP, DSM}). From an $[n,k,d]$ binary linear code $\C$, letting $N=2^k$ and using the same notation as in Theorem \ref{thm} for easy  comparison, one obtains an $N \times n$ matrix $\Phi_{\C}$ whose rows consist of all the codewords of $\C$ under the map $\psi$ (so that $\Phi_{\C}$ is a matrix of entries $\pm 1$). The probability space is choosing $p$ distinct rows uniformly at random from $\Phi_{\C}$ to form a submatrix $\Phi_n$. Then the sensing matrix $\frac{1}{\sqrt{n}}\Phi_{\C}$ is said to have the $(p,\delta,\ep)$-${\mathrm{StRIP}}$ if 
\begin{eqnarray} \label{1:strip} \P\left(\left\|\frac{1}{n}\Phi_n\Phi_n^*-I_p\right\|> \delta\right) < \ep. \end{eqnarray}
Equation (\ref{1:strip}) states that the event that all the eigenvalues of the matrix $\frac{1}{n}\Phi_n\Phi_n^*$ lie in the interval $[1-\delta, 1+\delta]$ has probability $1-\ep$. Since  the probability space in (\ref{1:strip}) is essentially the same as that in (\ref{SD}) of Theorem \ref{thm}, it can be seen that Theorem \ref{thm} provides a strong and much more precise description about how the eigenvalues of $\frac{1}{n}\Phi_n\Phi_n^*$ are distributed along the real line, but Theorem \ref{thm} falls short of proving (\ref{1:strip}). It seems possible to prove (\ref{1:strip}) by considering a slightly different normalization of the random matrix from linear codes as done in (\cite{ChanK}). We might come back to this question in the future. 
\end{remark}

For truly random matrices with i.i.d. entries, finding the rate of convergence has been a long-standing question, starting from \cite{ConI,ConII,Asym} in early 1990s. Great progress has been made in the last 10 years, culminating in achieving the optimal rate of convergence $n^{-1}$ where $n$ is the size of the matrix (see \cite{Local law,UnWig,UnCov}). The major technique is the use of the Stieltjes transform. In this paper we also use this technique.

The convergence rate problem for the empirical spectral distribution of large sample covariance random matrices has been studied for example in \cite{ConII, MP local law}, and in particular in \cite{MP local law} an optimal rate of convergence of order $n^{-1}$ was obtained under quite general conditions. However, despite our best effort, none of the techniques in \cite{ConII} and \cite{MP local law} can be easily applied directly to our setting. Instead we use a combination of ideas from \cite{ConII} and \cite{MP local law}. Moreover, it is not clear to us what the best rate of convergence is under general linear codes in terms of dual distance. It might be interesting to find out a general sufficient condition on the dual distance that guarantees the optimal rate of convergence $n^{-1}$. We hope to stress this problem in the future.






The paper is now organized as follows. In Section \ref{pre}, \textbf{Preliminaries} we introduce the main tool, the Stieltjes transform and related formulas and lemmas which will play important roles in the Proof of Theorem \ref{thm}. In Section \ref{proof}, we show how Theorem \ref{thm} can be derived directly from a major statement in terms of the Stieltjes transform (Theorem \ref{thm2}). While the argument is standard, it is quite technical and non-trivial. To streamline the idea of the proof, we put some of the arguments in Section \ref{Appen} \textbf{Appendix}. In Section \ref{proof2}, we give a detailed proof of Theorem \ref{thm2}. 

\section{Preliminaries}\label{pre}



\subsection{Stieltjes Transform}
In this section we recall some basic knowledge of Stieltjes transform. Interested readers may refer to \cite[Chapter B.2]{SA} for more details.

Let $F$ be an arbitrary real function with bounded variation, and $\mu$ be the corresponding (signed) measure. The Stieltjes transform of $F$ (or $\mu$) is defined by
$$s(z):=\int_{-\infty}^\infty \frac{\d F(x)}{x-z}=\int_{-\infty}^\infty \frac{\mu(\d x)}{x-z},$$
where $z$ is a complex variable outside the support of $F$ (or $\mu$). In particular, $s(z)$ is well-defined for all $z \in \mathbb{C}^+:=\{z \in \mathbb{C}: \Im z > 0\}$, the upper half complex plane. Here $\Im z$ is the imaginary part of $z$.

It can be verified that $s(z) \in \mathbb{C}^+$ for all $z \in \mathbb{C}^+$. The complex variable $z$ is commonly written as $z=E+\i\eta$ for $E,\eta \in \mathbb{R}$.

The Stieltjes transform is useful because a function of bounded variation (or signed measures) can be recovered from its Stieltjes transform via the inverse formula (\cite{ConI, Asym}):
$$\mu((x_1,x_2])=F(x_2)-F(x_1)=\lim_{\eta \downarrow 0}\frac{1}{\pi}\int_{x_1}^{x_2} \Im(s(E+\i\eta))\d E.$$
Here $\eta \downarrow 0$ means that the real number $\eta$ approaches zero from the right. Moreover, unlike the method of moments, the convergence of Stieltjes transform is both necessary and sufficient for the convergence of the underlying distribution (see \cite[Theorem B.9]{SA}).

\subsection{Resolvent Identities and Formulas for Green function entries}


Let $X=(X_{jk})$ be a $p \times n$ matrix. Denote by $G$ the Green function of $XX^*$, that is,
$$G:=G(z)=(XX^*-zI)^{-1},$$
where $z \in \mathbb{C}^+$ and $I$ is the identity matrix.

Given a subset $T \subset [1\isep p]:=\{1,2,\cdots,p\}$, let $X^{(T)}$ be the $p \times n$ matrix whose $(j,k)$-th entry is defined by $(X^{(T)})_{jk}:=\mathbbm{1}_{j \notin T}X_{jk}$. In addition, let $G^{(T)}$ be the Green function of $X^{(T)}X^{(T)*}$. We write $\R$ and $\R^{(T)}$ as the Green functions of $X^*X$ and $X^{(T)*}X^{(T)}$ respectively. Then for $\l \in [1\isep p] \setminus T$, we have \cite[(3.8)]{MP local law}
\begin{equation}\label{diagonal}
\frac{1}{G_{\l\l}^{(T)}}=-z-z\sum_{j,k}X_{\l j}\R_{jk}^{(T\l)}\overline{X}_{\l k},
\end{equation}
where the indices $j,k$ vary in $[1\isep n]$, and $\R_{jk}^{(T\l)}$ is the $(j,k)$-th entry of the matrix $\R^{(T \cup \{\l\})}$.

The two Green functions $G^{(T)}$ and $\mathcal{R}^{(T)}$ are related by the following identity (\cite[Lemma 3.9]{MP local law}):
\begin{equation}\label{Tr GR}
\Tr G^{(T)}-\Tr \R^{(T)}=\frac{n-(p-|T|)}{z}.
\end{equation}
Here $|T|$ is the cardinality of the set $T$, and $\Tr A$ is the trace of the matrix $A$.

Recall that we denote $\eta=\Im z$. Then we have the following eigenvalue interlacing property (\cite[Lemma 3.10]{MP local law})
\begin{equation}\label{Interlacing}
|\Tr G^{(T)}-\Tr G| \leq C\eta^{-1},
\end{equation}
where $C$ is a constant depending on the set $T$ only, and also the Wald's identity (see \cite[(3.14)]{MP local law} or \cite[(3.6)]{Local law})
\begin{equation}\label{Wald}
\sum_k |\R_{jk}^{(T)}|^2=\eta^{-1}\Im \R_{jj}^{(T)}.
\end{equation}

\subsection{Stieltjes Transform of the Marchenko-Pastur Law}
The Stieltjes transform $s_{\MP}$ of the Marchenko-Pastur distribution given in (\ref{mppdf}) can be computed as (see \cite{ConII})
\begin{equation}\label{SMP}
s_{\MP}(z)=-\frac{y+z-1 - \sqrt{(y+z-1)^2-4yz}}{2yz}.
\end{equation}
It is well-known that $s_{\MP}(z)$ is the unique function that satisfies the equation of $u(z)$ in
\begin{equation}\label{SMP2}
u(z)=\frac{1}{1-y-z-yzu(z)}
\end{equation}
such that $\Im u(z) > 0$ whenever $\eta:=\Im z > 0$.

If a function $f: \mathbb{C}^+ \to \mathbb{C}^+$ satisfies Equation \ref{SMP2} with a small perturbation, we then expect that $f(z)$ should be quite close to $s_{\MP}(z)$ as well. This is quantified by the following result. First, we define
\begin{equation}\label{kappa}
\kappa:=\min\{|E-a|,|E-b|\}
\end{equation}
where $a$ and $b$ are constants given in (\ref{ab}) and for a fixed constant $\tau > 0$, we define
\begin{equation}\label{st}
\st:=\bigg\{z=E+\i\eta: \kappa \leq \tau^{-1}, n^{-1/4+\tau} \leq \eta \leq \tau^{-1}
\bigg\}.
\end{equation}
\begin{lemma}\cite[Lemma 4.5]{MP local law}\label{diff}
Suppose the function $\delta: \st \to (0,\infty)$ satisfies:
\begin{enumerate}
\item $n^{-2} \leq \delta(z) \leq \ep$ for some fixed constant $\ep > 0$ for all $z \in \st$;
\item $\delta$ is Lipschitz continuous with Lipschitz constant $n$;
\item for each fixed $E$, the function $\eta \mapsto \delta(E+\i\eta)$ is nonincreasing for $\eta > 0$.
\end{enumerate}
Suppose $u: \st \to \mathbb{C}$ is the Stieltjes transform of a probability measure satisfying
\begin{equation}\label{perturb}
u(z)=\frac{1}{1-y-z-yzu(z)+\Delta(z)}
\end{equation}
for some $\Delta(z)$.

Fix $z \in \st$ and define $L(z):=\left\{w \in \st: \Re w=\Re z, \Im w \in [\Im z,1]\cap (n^{-5}\mathbb{N}) \right\}$, where $\Re z$ is the real part of $z$. Suppose that
\begin{equation}\label{Delta}
|\Delta(w)| \leq \delta(w), \quad \forall \, w \in L(z) \cup \{z\}.
\end{equation}
Then we have
$$|u(z)-s_{\MP}(z)| \leq \frac{C\delta(z)}{\sqrt{\kappa+\eta+\delta(z)}},$$
where $\kappa$ is the $z$-dependent variable defined as in (\ref{kappa}).
\end{lemma}
\subsection{Convergence of Stieltjes Transform in Probability}
The following result is useful to bound the convergence rate of a random Stieltjes transform in probability.
\begin{lemma}\label{deviation}
Let $\M$ be a $p \times n$ random matrix with independent rows, $S=\M\M^*$, and $m(z)$ be the Stieltjes transform of the empirical spectral distribution of $S$. Then
$$\P\left(|m(z)-\E m(z)| \geq r \right) \leq 2\exp\left(-\frac{p\eta^2r^2}{8}\right).$$
\end{lemma}
\begin{proof}[Proof of Lemma \ref{deviation}]
Note that the $(j,k)$-th entry of $S$ is simply the inner product of the $j$-th and $k$-th rows of $\M$. Hence varying one row of $\M$ only gives an additive perturbation of $S$ of rank at most two. Applying the resolvent identity \cite[(2.3)]{Local law}, we see that the Green function is also only affected by an additive perturbation by a matrix of rank of at most two and operator norm at most $2\eta^{-1}$. Then the desired result follows directly by applying the McDiarmid's Lemma \cite[Lemma F.3]{Local law}.
	
\end{proof}
For the purpose of this paper, we define an $n$-dependent event $\Xi$ to hold \emph{with high probability} if for any $D>0$, there is a quantity $n_0=n_0(D)>0$ such that $\P(\Xi) \geq 1-n^{-D}$ for any $n > n_0$.

\section{Proof of Theorem \ref{thm}}\label{proof}

From this section onwards, let $\C$ be a linear code of length $n$ over $\F_q$ with dual distance $\db \geq 5$. Let $\psi: \F_q \to \mathbb{C}^\times$ be the standard additive character, extended to $\F_q^n$ component-wisely. Write $\D=\psi(\C)$.

Let $\Phi_n$ be a $p \times n$ random matrix whose rows are picked from $\D$ uniformly and independently. This makes $\D^p$ a probability space. Let $y:=p/n \in (0,1)$ be fixed. Write $X=n^{-1/2}\Phi_n$ and $\G=XX^*$ the Gram matrix of $X$. Furthermore, let $\mu_n$ be the empirical spectral measure of $\G$ given by (\ref{mun}).

Denote $s_{\G}(z)$ to be the Stieltjes transform of $\mu_n$, which is given by
$$s_{\G}(z)=\frac{1}{p}\sum_{j=1}^p \frac{1}{\lambda_j-z}=\frac{1}{p}\Tr G,$$
where $\lambda_1,\cdots,\lambda_p$ are the eigenvalues of the matrix $\G$, and $G$ is the Green function of $\G$, that is, $G:=G(z)=(\G-zI)^{-1}$. Note that in this setting this Stieltjes transform $s_{\G}(z)$ is itself a random variable.

Denote
\begin{equation}\label{EStietjes}
s_n(z):=\E s_{\G}(z)=\frac{1}{p} \, \E \Tr G.
\end{equation}
Here $\E$ is the expectation with respect to the probability space $\D^p$.


\subsection{An equation for $s_n(z)$}\label{esn}
In the following result, we write $s_n(z)$ defined in (\ref{EStietjes}) in the form of the equation (\ref{SMP2}) with a small perturbation.

\begin{theorem}\label{thm2}
For any $z \in \st$,
$$s_n(z)=\frac{1}{1-y-z-yzs_n(z)+\Delta(z)}$$
where $\Delta(z)=O(n^{-1}\eta^{-3})$.
\end{theorem}
We remark that Theorem \ref{thm2} is a major technical result regarding the expected Stieltjes transform $s_n(z)$, from which Theorem \ref{thm} can be derived directly without reference to linear codes at all. The proof of Theorem \ref{thm2} is, however, quite complicated and is directly related to properties of linear codes. To streamline the idea of the proof, here we assume Theorem \ref{thm2} and sketch a proof of Theorem \ref{thm}. The proof of Theorem \ref{thm2} is postponed to Section \ref{proof2}.
\subsection{Proof of Theorem \ref{thm}}\label{sec}
Assuming Theorem \ref{thm2}, we can first estimate the term $|s_{\G}(z)-s_{\MP}(z)|$, following ideas from \cite{Local law} and \cite{MP local law}.

\begin{theorem}\label{Stieltjes}
Assume that Theorem \ref{thm2} holds. Then for any fixed $z \in \st$, we have
$$|s_{\G}(z)-s_{\MP}(z)| \leq n^\tau(n^{-1/4}+n^{-1}\eta^{-7/2})$$ with high probability.
\end{theorem}




\begin{proof}[Proof of Theorem \ref{Stieltjes}]
We can check that all the conditions of Lemma \ref{diff} are satisfied: first by Theorem \ref{thm2} we see that (\ref{perturb}) holds for $u(z)=s_n(z)$; in addition, (\ref{Delta}) holds for $\delta(z)=\delta(E+\i\eta)=Cn^{-1}\eta^{-3}$, and this function is independent of $E$, nonincreasing in $\eta>0$ and Lipschitz continuous with Lipschitz constant $Cn^{-4\tau} < n$. Hence by Lemma \ref{diff}, we have
$$|s_n(z)-s_{\MP}(z)| \leq \frac{C\delta(z)}{\sqrt{\kappa+\eta+\delta(z)}}.$$
Note that in $\st$ we have $\delta(z)=O(n^{-1/4-3\tau})=o(\eta)$. Therefore we have	 \begin{equation}\label{EStieltjes}
|s_n(z)-s_{\MP}(z)|=O(n^{-1}\eta^{-7/2})
\end{equation}
for all $z \in \st$.

Now Lemma \ref{deviation} implies that
$$\P \left(|s_{\G}(z)-s_n(z)| > n^{\tau-1/4} \right) \leq 2\exp\left(-\frac{yn(n^{\tau-1/4})^4}{8}\right)= 2\exp\left(-\frac{yn^{4\tau}}{8}\right) \leq n^{-D}$$
on $\st$, for any $D > 0$ and large enough $n$. Combining this with (\ref{EStieltjes}) completes the proof of Theorem \ref{Stieltjes}.
\end{proof}

Finally, armed with Theorem \ref{Stieltjes}, we can derive Theorem \ref{thm} from a standard application of the Helffer-Sj\"{o}strand formula in random matrix theory. The argument is essentially complex analysis. Interested readers may refer to Section \ref{Appen} \textbf{Appendix} for details.

\section{Proof of Theorem \ref{thm2}}\label{proof2}

Now we give a detailed proof of Theorem \ref{thm2}, in which the condition that $\db \geq 5$ becomes essential.

\subsection{Linear codes with $\db \geq 5$}
Recall the notation from the beginning of Section \ref{proof}. Let $\C$ be a linear code of length $n$ over $\F_q$. First is a simple orthogonality result regarding $\C$.
\begin{lemma}\label{lem}
Let $\mathbf{a} \in \F_q^n$. Then
$$\frac{1}{\#\C}\sum_{\mathbf{c} \in \C}\psi(\mathbf{a}\cdot \mathbf{c})=\begin{cases}
1 &(\mathbf{a} \in \Cb),\\
0 &(\mathbf{a} \notin \Cb).
\end{cases}$$
Here $\mathbf{a}\cdot\mathbf{c}$ is the usual inner product between the vectors $\mathbf{a}$ and $\mathbf{c}$.
\end{lemma}
As in Section \ref{proof}, let $\Phi_n$ be a $p \times n$ random matrix whose rows are picked from $\D=\psi(\C)$ uniformly and independently and let $X=n^{-1/2}\Phi_n$. Denote by $X_{jk}$ the $(j,k)$-th entry of $X$.
\begin{corollary}\label{cor}
Assume $\db \geq 5$. Then for any $\l \in [1\isep p]$,

(a) $\E(X_{\l j}\overline{X}_{\l k})=0$ if $j \neq k$;

(b) $\E(X_{\l j}X_{\l t}\overline{X}_{\l k}\overline{X}_{\l s})=0$ if the indices $j,t,k,s$ do not come in pairs. If the indices come in pairs, then $|\E(X_{\l j}X_{\l t}\overline{X}_{\l k}\overline{X}_{\l s})| \leq n^{-2}$.

Here $\E$ is the expectation with respect to the probability space $\D^p$.
\end{corollary}
\begin{proof}[Proof of Corollary \ref{cor}]

For simplicity, denote by $\mathbf{e}_i=(0,\cdots,0,1,0,\cdots,0) \in \F_q^n$ the vector with a $1$ at the $i$-th entry and $0$ at all other places.

(a) It is easy to see that
\begin{eqnarray*} \E(X_{\l j}\overline{X}_{\l k})&=&n^{-1}(\#\C)^{-1}\sum_{\mathbf{c} \in \C}\psi(c_j-c_k)\\
 &=& n^{-1}(\#\C)^{-1} \sum_{\mathbf{c} \in \C}\psi \left((\mathbf{e}_j-\mathbf{e}_k) \cdot \mathbf{c} \right).
 \end{eqnarray*}
As $\db \geq 5$ and $j \ne k$, so $0 \ne \mathbf{e}_j-\mathbf{e}_k \notin \Cb$, and the desired result follows directly from Lemma \ref{lem}.

(b) Again we can check that
\[ \E(X_{\l j}X_{\l t}\overline{X}_{\l k}\overline{X}_{\l s})=n^{-2}(\#\C)^{-1} \sum_{\mathbf{c} \in \C}\psi\left((\mathbf{e}_j+\mathbf{e}_t-\mathbf{e}_k-\mathbf{e}_s) \cdot \mathbf{c} \right).\]
If the indices $j,t,k,s$ do not come in pairs, since $d^\bot \ge 5$, we have $0 \ne \mathbf{e}_j+\mathbf{e}_t-\mathbf{e}_k-\mathbf{e}_s \notin \Cb$, and the result is zero by Lemma \ref{lem}; If the indices $j,t,k,s$ do come in pairs, noting that $|X_{jk}|=n^{-1/2}$, we also obtain the desired estimate. This completes the proof of Corollary \ref{cor}.
\end{proof}

\subsection{Resolvent identities}
We start with the resolvent identity (\ref{diagonal}) for $T=\emptyset$. The sum on the right of (\ref{diagonal}) can be written as
$$z\sum_{j,k} X_{\l j}\R_{jk}^{(\l)}\overline{X}_{\l k}=\frac{z}{n}\sum_j \R_{jj}^{(\l)}+\Zl,$$
where
\begin{equation}\label{Zl}
\Zl=z\sum_{j \neq k}X_{\l j}\R_{jk}^{(\l)}\overline{X}_{\l k}.
\end{equation}
Using (\ref{diagonal}) and (\ref{Tr GR}) we have
\begin{align}
\frac{1}{G_{\l\l}}&=-z-\frac{z}{n}\Tr  \R^{(\l)}-\Zl\nonumber\\
&=-z-\frac{z}{n}\left(\Tr
\Gl-\frac{n-p+1}{z}\right)-\Zl\nonumber\\
&=1-y-z-yzs_n(z)+\Yl,\label{diagonal2}
\end{align}
where
\begin{align}
\Yl &= yzs_n(z)-\frac{z}{n}\Tr \Gl+\frac{1}{n}-\Zl \nonumber \\
&= \frac{z}{n}\left(\E \Tr G-\Tr \Gl \right)+\frac{1}{n}-\Zl. \label{Yl}
\end{align}
\subsection{Estimates of $\Zl$ and $\Yl$}
We now give estimates on the ($z$-dependent) random variable $\Zl$. First, given $T \subset [1\isep p]$, we denote $\E^{(T)}(\cdot):=\E(\cdot|X^{(T)})$.
\begin{lemma}\label{Zl2}
For any $\l \in [1\isep p]$, we have

(a) $\E^{(\l)}\Zl=\E\Zl=0$;

(b) $\E|\Zl|^2=O(n^{-1}\eta^{-2})$.
\end{lemma}
\begin{proof}[Proof of Lemma \ref{Zl2}]
(a) From the definition of $\Zl$ in (\ref{Zl}), we have
$$\E^{(\l)}\Zl=z\sum_{j \neq k}\R_{jk}^{(\l)}\E(X_{\l j}\overline{X}_{\l k})=0,$$
where the first equality follows from the fact that rows of $X$ are independent, and second equality follows from statement (a) of Corollary \ref{cor}. The proof of the result on $\E\Zl$ is similar by replacing $\R_{jk}^{(\l)}$ with $\E\R_{jk}^{(\l)}$.

(b) Expanding $|\Zl|^2$ and taking expectation $\E$ inside, noting that the rows of $X$ are independent, we have
\begin{align*}
\E|\Zl|^2&=|z|^2\E\left|\sum_{j \neq k}X_{\l j}\R_{jk}^{(\l)}\overline{X}_{\l k}\right|^2\\
&=|z|^2\sum_{\substack{j\neq k\\s\neq t}}\E\left(\R_{jk}^{(\l)}\overline{\R}_{st}^{(\l)}\right)\E(X_{\l j}X_{\l t}\overline{X}_{\l k}\overline{X}_{\l s}).
\end{align*}
Since $\db \geq 5$, by using statement (b) of Corollary \ref{cor} and Wald's identity (\ref{Wald}), together with the trivial bound $|\R_{jj}^{(\l)}| \leq \eta^{-1}$, we obtain
\begin{align*}
\E|\Zl|^2&\leq \frac{C|z|^2}{n^2}\sum_{j,k}\E|\R_{jk}^{(\l)}|^2\\
&=\frac{C|z|^2}{n^2\eta}\sum_j\E\Im \R_{jj}^{(\l)} \leq\frac{C}{n\eta^2}.
\end{align*}
Here $C$ is a generic absolute constant which may be different in each occurrence.
\end{proof}
The above estimations lead to the following estimations about $\Yl$.
\begin{lemma}\label{Yl2}
For any $\l \in [1\isep p]$, we have

(a) $\E \Yl=O(n^{-1}\eta^{-1})$;

(b) $\E|\Yl|^2=O(n^{-1}\eta^{-2})$.
\end{lemma}
\begin{proof}[Proof of Lemma \ref{Yl2}]
(a) By (\ref{Yl}) we get
$$\E \Yl=\frac{z}{n}\E(\Tr G-\Tr \Gl)+\frac{1}{n}-\E \Zl=\frac{z}{n}\E(\Tr G-\Tr \Gl)+\frac{1}{n},$$
where the second equality follows from (a) of Lemma \ref{Zl2}. Using (\ref{Interlacing}) we easily obtain
$$|\E \Yl| \leq \frac{C|z|}{n\eta} \leq \frac{C}{n\eta}.$$
(b) We split $\E|\Yl|^2$ as
\begin{equation}\label{Yl3}
\E|\Yl|^2=\E|\Yl-\E\Yl|^2+|\E\Yl|^2=V_1+V_2+|\E\Yl|^2,
\end{equation}
where
$$V_1=\E|\Yl-\E^{(\l)}\Yl|^2, \quad V_2=\E|\E^{(\l)}\Yl-\E\Yl|^2.$$
We first estimate $V_1$. By the definition of $\Yl$ in (\ref{Yl}) and applying (a) of Lemma \ref{Zl2}, we see that
$$\Yl-\E^{(\l)}\Yl=-\Zl+\E^{(\l)}\Zl=-\Zl.$$
Then by (b) of Lemma \ref{Zl2} we obtain
\begin{equation}\label{V1}
V_1=\E|\Zl|^2=O(n^{-1}\eta^{-2}).
\end{equation}

Next we estimate $V_2$. Again by (\ref{Yl}) and Lemma \ref{Zl2}, we have
$$\E^{(\l)}\Yl-\E\Yl=-\frac{z}{n}(\Tr \Gl-\E\Tr \Gl)-(\E^{(\l)}\Zl-\E\Zl)=-\frac{z}{n}(\Tr \Gl-\E\Tr \Gl).$$

Hence
\begin{align}\label{V2}
V_2&=\frac{|z|^2}{n^2}\E|\Tr \Gl-\E\Tr \Gl|^2\nonumber\\
&=\frac{|z|^2}{n^2}\sum_{m \neq \l}\E|\E^{(T_{m-1})}\Tr \Gl-\E^{(T_m)}\Tr \Gl|^2,
\end{align}
where $T_0:=\emptyset$ and $T_m:=[1\isep m]$ for $m \in [1\isep p]$. The second equality follows from applying successively the law of total variance to the rows of $\Phi_n$.

For $m \neq \l$, denote $\gamma_m:=\E^{(T_{m-1})}\Tr \Gl-\E^{(T_m)}\Tr \Gl$ and $\sigma_m:=\Tr \Gl-\Tr G^{(\l,m)}$. It is easy to check that
$$\gamma_m=\E^{(T_{m-1})}\sigma_m-\E^{(T_m)}\sigma_m.$$
Thus by (\ref{Interlacing}) we have $|\gamma_m| \leq C\eta^{-1}$.

Putting this into (\ref{V2}) yields
$$V_2 \leq \frac{C|z|^2}{n\eta^2} \leq \frac{C}{n\eta^2}.$$
Plugging the estimates of $\E\Yl$ in statement (a), $V_1$ in (\ref{V1}) and $V_2$ above into the equation (\ref{Yl3}), we obtain the desired estimate of $\E|\Yl|^2$. This finishes the proof of Lemma \ref{Yl2}.
\end{proof}
We can now complete the proof of Theorem \ref{thm2}.
\begin{proof}[Proof of Theorem \ref{thm2}]
Taking reciprocal and then expectation on both sides of (\ref{diagonal2}), we get
\begin{equation}\label{diagonal3}
\E G_{\l\l}=\E\frac{1}{\an+\Yl}=\frac{1}{\an}+\Al=\frac{1}{\an+\dl},
\end{equation}
where
$$\an=1-y-z-yzs_n(z),$$
\begin{equation}\label{Al}
\Al=\E\frac{1}{\an+\Yl}-\frac{1}{\an}=-\frac{1}{\an^2}\E \Yl+\frac{1}{\an^2}\E\frac{\Yl^2}{\an+\Yl},
\end{equation}
and
\begin{equation}\label{dl}
\dl=\left(\frac{1}{\an}+\Al\right)^{-1}-\an=-\frac{\an^2 \Al}{1+\an \Al}.
\end{equation}
Multiplying $\an^2$ on both sides of (\ref{Al}) and using the estimate $|\an+\Yl| \geq \eta$, we obtain
\begin{equation}\label{Al2}
|\an^2 \Al|=\left|-\E \Yl+\E\frac{\Yl^2}{\an+\Yl}\right|\leq |\E \Yl|+\frac{1}{\eta}\E|\Yl|^2.
\end{equation}

Putting the results of Lemma \ref{Yl2} in (\ref{Al2}), we get
$$|\an^2\Al| \leq \frac{C}{n\eta^3}.$$
Using the fact that $|\an|^{-1}\leq \eta^{-1}$, we have $|\an\Al|\leq Cn^{-1}\eta^{-4}$, so that $|1+\an\Al| \geq C$.

Substituting all these into (\ref{dl}) yields
$$|\dl| \leq \frac{C}{n\eta^3}$$
for all $z \in \st$.

Then the theorem follows directly from summing both sides of (\ref{diagonal3}) for all $\l \in [1\isep p]$ and then dividing both sides by $p$.
\end{proof}



\section{Appendix}\label{Appen}
In this section, we use Helffer-Sj\"{o}strand formula to prove Theorem \ref{thm} from Theorem \ref{Stieltjes}. This is a standard procedure well-known in random matrix theory. We follow the idea based on \cite[Appendix C]{Local law}.

First we define the signed measure $\hmu$ and its Stieltjes transform $\hs$ by
$$\hmu:=\mu_n-\rmp, \hs(z):=\int \frac{\hmu(dx)}{x-z}=s_{\G}(z)-s_{\MP}(z).$$
Now fix $\ep \in (0,1/4)$ and define $\weta:=n^{-1/4+\ep/2}$. For any interval $\I \subset [a-1,b+1]$, where $a$ and $b$ are constants defined in (\ref{ab}), we choose a smoothed indicator function $f \equiv f_{\I,\weta} \in C_c^\infty(\mathbb{R}; [0,1])$ satisfying $f(u)=1$ for $u \in \I$, $f(u)=0$ for $\mathrm{dist}(E,\I) \geq \weta, \|f'\|_\infty \leq C\weta^{-1}$, and $\|f''\|_\infty \leq C\weta^{-2}$. These imply that the supports of $f'$ and $f''$ have Lebesgue measure bounded by $2\weta$. In addition, choose a smooth even cutoff function $\chi \in C_c^\infty(\mathbb{R}; [0,1])$ with $\chi(v)=1$ for $|v| \leq 1, \chi(v)=0$ for $|v| \geq 2$ and $\|\chi'\|_\infty \leq C$. Throughout this section, $C$ represents a positive constant whose value may vary in each appearance.

Then by the Helffer-Sj\"{o}strand formula, we get
$$\int f(\lambda)\hmu(\d\lambda)=\frac{1}{2\pi}\iint(\pa_u+\i\pa_v)[f(u)+\i vf'(u))\chi(v)]\hs(u+\i v)\d v\d u.$$
As LHS is real, we can write as
\begin{align}
\int f(\lambda)\hmu(\d\lambda)&=-\frac{1}{2\pi}\int\int_{|v| \leq \weta}f''(u)\chi(v)v\Im\hs(u+\i v)\d v\d u\label{I1}\\
&-\frac{1}{2\pi}\int\int_{|v| > \weta}f''(u)\chi(v)v\Im\hs(u+\i v)\d v\d u\label{I2}\\
&+\frac{\i}{2\pi}\iint(f(u)+\i vf'(u))\chi'(v)\hs(u+\i v)\d v\d u\label{I3}
\end{align}
First, by the trivial identity $\hs(u-\i v)=\overline{\hs(u+\i v)}$ and the fact that $\hs$ is Lipschitz continuous on the compact set $\mathbf{S}_{\ep/2}$, we can easily extend Theorem \ref{Stieltjes} as follows:
\begin{lemma}\label{lem2}
For any fixed $\ep > 0$, we have, with high probability,
$$|\hs(u+\i v)| \leq n^{\ep/2}(n^{-1/4}+n^{-1}|v|^{-7/2}),$$
for all $u,v \in \mathbb{R}$ such that $\min\{|u-a|,|u-b|\} \leq 2\ep^{-1}$ and $|v| \in [\weta,2\ep^{-1}]$.
\end{lemma}
We may now estimate the three terms appearing in (\ref{I1})-(\ref{I3}). First, for the term in (\ref{I3}), by using the fact that $\chi$ is even with support in $[-2,2]\setminus (-1,1)$, we have
\begin{equation}
\left|\iint (f(u)+\i vf'(u))\chi'(v)\hs(u+\i v)\d v\d u\right| \leq C\weta
\end{equation}
with high probability.

We next estimate the term in (\ref{I1}). Since $v$ is small, we cannot apply Lemma \ref{lem2} directly. However it can be proved that for all $u$, the function $v \mapsto v\Im(s_{\G}(u+\i v))$ is nondecreasing for $v > 0$. This implies, for $v \in (0,\weta)$,
$$v\Im\hs(u+\i v) \leq v\Im s_{\G}(u+\i v) \leq \weta\Im s_{\G}(u+\i\weta) \leq \weta[n^{\ep/2}(n^{-1/4}+n^{-1}\weta^{-7/2})+C] \leq C\weta$$
with high probability.

Hence we have
\begin{equation}
\left|\int_\mathbb{R}\int_{|v| \leq \weta} f''(u)\chi(v)v\Im(\hs(u+\i v))\d v\d u\right| \leq \weta^{-1}\int_{|v|\leq \weta}C\weta\d v\leq C\weta .
\end{equation}

For the term in (\ref{I2}), we have
\begin{align*}
&\int_\mathbb{R}\int_{|v| > \weta} f''(u)\chi(v)v\Im(\hs(u+\i v))\d v\d u\\
&= \left[\int_{|v| > \weta} f'(u)\chi(v)v\Im(\hs(u+\i v))\d v\right]_{u=\inf \supp f''}^{u=\sup \supp f''}-\int_{\supp f''}\int_{|v| > \weta} f'(u)\chi(v)v\frac{\pa}{\pa u}\Im(\hs(u+\i v))\d v\d u
\end{align*}
The first term is zero. As for the second, by Cauchy-Riemann equation, we have
\begin{align*}
&\int_\mathbb{R}\int_{|v| > \weta} f''(u)\chi(v)v\Im(\hs(u+\i v))\d v\d u\\
&=\int_{\supp f''}\int_{|v| > \weta} f'(u)\chi(v)v\frac{\pa}{\pa v}\Re(\hs(u+\i v))\d v\d u\\
&=2\left[\int_{\supp f''} f'(u)\chi(v)v\Re(\hs(u+\i v))\d v\right]_{\weta}^\infty-2\int_{\supp f''}\int_{\weta}^\infty f'(u)(\chi'(v)v+\chi(v))\Re(\hs(u+\i v))\d v\d u
\end{align*}
For the first term, we get
$$\left|\left[\int_{\supp f''} f'(u)\chi(v)v\Re(\hs(u+\i v))\d u\right]_{\weta}^\infty\right|=\left|\int_{\supp f''} f'(u)\chi(\weta)\weta\Re(\hs(u+\i \weta))\d u\right| \leq C\weta$$
For the second term, we get
\begin{align*}
\left|\int_{\supp f''}\int_{\weta}^\infty f'(u)(\chi'(v)v+\chi(v))\Re(\hs(u+\i v))\d v\d u\right|&\leq Cn^{\ep/2}\int_{\weta}^2[n^{-1/4}(v+1)+n^{-1}(v^{-5/2}+v^{-7/2})]\d v\\
&\leq Cn^{\ep/2}(n^{-1/4}+n^{-1}\weta^{-5/2})\\
&\leq C\weta
\end{align*}
Putting all together, we get
$$\left|\int f(\lambda)\hmu(\d\lambda)\right| \leq C\weta$$
with high probability.

Now we have to return from the smooth function $f$ to the indicator function of $\I$. If $\I \subset [a-1,b+1]$, then we get
$$\mu_n(\I) \leq \int f_{\I,\weta}(\lambda)\mu_n(\d\lambda)\leq \int f_{\I,\weta}(\lambda)\rmp(\d\lambda)+C\weta\leq \rmp(\I)+C\weta$$
with high probability. On the other hand, denote by $\I':=\{x \in \mathbb{R}: \mathrm{dist}(x,\I^c) \geq \weta\}$ (which is hence a subset of $\I$), then we also have
$$\mu_n(\I) \geq \int f_{\I',\weta}(\lambda)\mu_n(\d\lambda)\geq \int f_{\I',\weta}(\lambda)\rmp(\d\lambda)-C\weta \geq \rmp(\I)-C\weta$$
with high probability. Hence
$$|\hmu(\I)| \leq Cn^{-1/4+\ep/2} \leq n^{-1/4+\ep}$$
with high probability. As $\ep \in (0,1/4)$ is arbitrary, we conclude that $\hmu(\I)=O_\prec(n^{-1/4})$ for any $\I \subset [a-1,b+1]$.

Then for a general interval $\I \subset \mathbb{R}$, we first note that we have proved that $\hmu([a,b])=O_\prec(n^{-1/4})$. As $\mu_n$ is a probability measure and $\rmp([a,b])=1$, we deduce that $\mu_n(\mathbb{R} \setminus [a,b])=O_\prec(n^{-1/4})$. Therefore we have
$$\mu_n(\I)=\mu_n(\I \cap [a,b])+\mu_n(\I \setminus [a,b])=\mu_n(\I \cap [a,b])+O_\prec(n^{-1/4})=\rmp(\I)+O_\prec(n^{-1/4})$$
where in the last step we use that $\hmu(\I)=O_\prec(n^{-1/4})$ for $\I \subset [a,b]$. From the calculation it is easy to see that the above estimate holds simultaneously for all $\I$ (i.e. the constant absorbed by $O_\prec$ is independent of $\I$).

This completes the proof of Theorem \ref{thm}.



\section*{Acknowledgments}
The second author would like to thank David Forney, Rob Calderbank and Neil Sloane for stimulating discussions. The third author would like to thank Zhigang Bao for comments and suggestions.

\end{document}